\documentclass[a4paper,11pt]{amsart}

\usepackage{amssymb,amsmath,amsfonts,amsthm,mathabx,mathrsfs,pdfsync,
enumerate,bbm,fancyhdr,enumerate,graphicx, url,afterpage,setspace,geometry,mathabx}

\usepackage[dvipsnames]{xcolor}
\usepackage[colorlinks=true, pdfstartview=FitV, linkcolor=blue, citecolor=blue, urlcolor=blue]{hyperref}

\usepackage[english]{babel}
\usepackage[utf8]{inputenc}

\DeclareMathOperator{\n1}{\mathbbm{1}}

\DeclareMathOperator{\Ker}{Ker}

\DeclareMathOperator{\Sp}{Sp}
\DeclareMathOperator{\C}{\mathbb{C}}
\DeclareMathOperator{\R}{\mathbb{R}}

\DeclareMathOperator{\N}{\mathbb{N}}

\DeclareMathOperator{\h}{\mathcal{H}}
\DeclareMathOperator{\K}{\mathcal{K}}

\DeclareMathOperator{\B}{\mathcal{B}}

\newtheorem{thm}{Theorem}

\newtheorem{lem}[thm]{Lemma}
\newtheorem{prop}[thm]{Proposition}

\newtheorem{cor}[thm]{Corollary}
\newtheorem{example}[thm]{Example}

\theoremstyle{definition}
\newtheorem{defn}[thm]{Definition}

\begin{document}

\title{Quasi-Hermitian pair and co-amenability}

\author{Bat-Od Battseren}
\email{batoddd@gmail.com}


\dedicatory{}


\begin{abstract}
We adapt the notion of quasi-Hermition group to the pairs $(G,H)$ of discrete group $G$ and its subgroup $H$. We show that a quasi-Hermitian pair is amenable in the sense of Eymard.
\end{abstract}

\maketitle
\setcounter{tocdepth}{1}
\section{Introduction}
Hermitian Banach $*$-algebras make an important class of $*$-algebras that extends the class of $C^*$-algebras. It has being an effectively elaborated in approximation theory, time-frequency analysis and signal process, and non-commutative geometry. In this paper, we are mainly focus on the (quasi-)Hermitian property of Banach $*$-algebras derived from a group.

Spectrum of a function in group algebra plays a crucial role in study of analytic group theory. For example, amenability is characterized by the fact that the spectrum of a normalized characteristic function has 1 as a spectral point in the reduced group $C^*$-algebra. The property of being Hermitian is a variant of such spectral study. A locally compact group $G$ is called \textit{Hermitian} (resp. \textit{quasi-Hermitian}) if the spectrum $\Sp(f,L^1(G))$ is in the real line  for all self-adjoint functions $f=f^*$ in $L^1(G)$ (resp. in $C_c(G)$). It was a long-standing conjecture that Hermitian groups are amenable until it was affirmatively proven in \cite{SW20} even for quasi-Hermitian groups. In this paper, we extend the notion of quasi-Hermition group to pairs $(G,H)$ of discrete group $G$ and its subgroup $H$ and adapt the Samei-Wiersma's results for quasi-Hermition pairs.

As amenability is characterized in many ways, it has many versions such as weak amenability, a-T-menability (also known as Haagerup property), Haagerup-Kraus's approximation property, co-amenability, etc. We are more interested in co-amenability in the sense of Eymard. Recall the original definition of amenability given by John von Neumann: A discrete group $G$ is \textit{amenable} if there is a $G$-invariant state on $\ell^\infty(G)$. Similarly, a subgroup $H$ of $G$ is \textit{co-amenable} in $G$ if there is a $G$-invariant state on $\ell^\infty (G/H)$. Our main result is as follows.

\begin{thm}\label{thm A}
If $(G,H)$ is quasi-Hermtian, then $H$ is co-amenable in $G$.
\end{thm}

The notion of quasi-Hermitian pair extends the usual notion of quasi-Hermitian group. In particular, if we take $H$ to be the trivial subgroup, we recover Samei-Wiersma's result for discrete groups. Indeed, $G$ is quasi-Hermitian if and only if $(G,\{e\})$ is quasi-Hermitian, and $G$ is amenable if and only if $\{e\}$ is co-amenable in $G$. This extension allows us to target a richer class.

About the proof of Theorem \ref{thm A}, we essentially follow the steps of \cite{SW20}. The first step is that we construct Banach $*$-algebras $PF^*_p(G:H)$ associated to the pair $(G,H)$ and $1\leq p\leq 2$. It generalizes the usual Banach $*$-algebras $PF^*_p(G)$. When $p=2$, we have $PF^*_p(G:H)=C^*_{\lambda_{G/H}}(G)$.  The second step is to see that  the triple 
\begin{align*}
(PF^*_{p_1}(G:H),PF^*_{p_2}(G:H),PF^*_{p_3}(G:H))
\end{align*} is a spectral interpolation of $*$-semisimple Banach $*$-algebras relative to $PF^*_1(G:H)$ for any $1\leq p_1<p_2< p_3\leq 2$. Then in Theorem \ref{thm char q-Her}, we will prove that $(G,H)$ is quasi-Hermitian if and only if the spectrum of $[f]\in [\C G]$ in $PF^*_p (G:H)$ does not depend on $1\leq p\leq 2$. This latter was proven in \cite{Bar90} for amenable Hermitian groups, during which time it was unknown that Hermitian groups are amenable. The third step is to see that if the universal $C^*$-enveloping algebra of $PF^*_1(G,H)$ is canonically isomorphic to $C^*_{\lambda_{G/H}} (G)$, then $H$ is co-amenable in $G$. Then Proposition 2.6 of \cite{SW20} concludes our proof.

\section{Preliminaries}
We recall some definitions and useful results.
\subsection{Banach $*$-algebras} Let $A$ be a Banach $*$-algebra and let $S$ be a $*$-subalgebra of $A$. For $a\in A$, we denote by $\Sp(a,A)$ the spectrum of $a$ in $A$, and by $r(a,A)$ the spectral radius of $a$ in $A$. Denote by $S_h$ the self-adjoint elements of $S$. We say that $S$ is \textit{quasi-Hermitian} in $A$ if $
\Sp(a,A)\subseteq \R$  for all $a\in S_h$.
 Similarly, we say that $S$ is \textit{quasi-symmetric} in $A$ if
$\Sp(a^*a,A)\subseteq \R_+$  for all $a\in S$.
Banach $*$-algebra $A$ is called \textit{Hermitian} (resp. \textit{symmetric}) if it is quasi-Hermitian (resp. quasi-symmetric) in itself. Shirali-Ford's theorem states that a Banach $*$-algebra is Hermitian if and only if it is symmetric (see for example \cite[Theorem 41.5]{BD73}).

By a \textit{$*$-representation} of $A$, we understand a pair $(\pi,\h)$ where $\h$ is a Hilbert space and $\pi:A\rightarrow \B(\h)$ is a $*$-homomophism into bounded operators on $\h$. It is a well known fact that any $*$-representation is norm decreasing. We say that $A$ is \textit{$*$-semisimple} if for any non-zero element $a\in A$, there is a $*$-representation that sends $a$ into a non-zero element. In the sequel, $A$ is always assumed to be $*$-semisimple. For a $*$-representation $(\pi,\h)$ of $A$, we denote by $C^*_\pi(A)$ the completion of $\pi(A)$ in $\B(\h)$. Let  $(\pi,\h)$ and $(\rho,\K)$ be  $*$-representations of $A$. We say that $\pi$ \textit{weakly contains} $\rho$ and write $\rho\prec \pi$ if $\|\rho(a)\|\leq \|\pi(a)\|$ for all $a\in A$. We say that $\pi$ and $\rho$ are \textit{weakly equivalent} and write $\pi\sim \rho$ if they weakly contains each other. We say that $\pi$ and $\rho$  are \textit{isomorphic} and write $\pi\cong\rho$ if there is a unitary operator $U\in \B(\h,\K)$ such that $\pi(a) = 
U^*\rho(a)U$ for all $a\in A$. Note that for any family $\{(\pi_i,\h)\}$ of $*$-representations, their $\ell^2$-direct-sum gives a $*$-representation $(\pi,\h)$  such that $\pi_i\prec \pi$ for all $a\in A$. This yields the existence of the \textit{universal $*$-representation} $(\pi_u,\h_u)$ of $A$. This is the unique $*$-representation satisfying the following universal property: Any $*$-representation of $A$ uniquely factors through $C^*_{\pi_u}(G)$. The $C^*$-algebra $C^*(A) = C^*_{\pi_u}(G)$ is called the \textit{universal $C^*$-enveloping algebra} of $A$. Since $A$ is $*$-semisimple, we can see $A$ as a dense $*$-subalgebra of $C^*(G)$. If $A$ is Hermitian, then the inclusion $A\subseteq C^*(A)$ is spectral, i.e. 
\begin{align*}
\Sp(a,A)=\Sp(a,C^*(A))\quad \text{for all} \quad a\in A.
\end{align*} The converse is rather trivial. 

The main tool applied in \cite{SW20} was the notion of spectral interpolation which was inspired by \cite{Pyt82}. Let us explain it more in detail. We say that $A\subseteq B$ is a \textit{nested pair} of ($*$-semisimple) Banach $*$-algebras if $A$ and $B$ are ($*$-semisimple) Banach $*$-algebras and $A$ embeds continuously into $B$ as a dense $*$-subalgebra. A nested triple of ($*$-semisimple) Banach algebras is defined similarly. Let $S$ be a $*$-subalgebra of $A$. We say that $S_h$ has
\textit{invariant spectral radius} in $(A,B)$ if $r(a,A)=r(a,B)$ for all $a \in S$. We say that $S$ is
\textit{spectral subalgebra} of $(A,B)$ if  $\Sp(a,A)=\Sp(a,B)$ for all $a \in S$. Suppose that $A\subseteq B\subseteq C$ is a nested triple of Banach $*$-algebras and $S$ is a dense $*$-subalgebra of $A$. We say that $(A,B,C)$ is a \textit{spectral interpolation} of triple Banach $*$-algebras \textit{relative} to $S$ if there is a constant $0<\theta<1$ such that
\begin{align*}
r(a,B)\leq r(a,A)^{1-\theta} r(a,C)^{\theta}\quad \text{for all} \quad a\in S_h.
\end{align*}
The following two results are crucial to our main result.

\begin{prop}[{\cite[Proposition 2.6]{SW20}}]\label{prop 2.6}
Let $A\subseteq B$ be a nested pair of $*$-semisimple Banach $*$-algebras, and $S$ be a dense $*$-subalgebra of $A$. Suppose that $S_h$ has an invariant spectral  radius in $(A,B)$. Then $A$ and $B$ have the same universal $C^*$-enveloping algebra. In particular, if $B$ is a $C^*$-algebra, then $B$ is the universal $C^*$-enveloping algebra of $A$.
\end{prop}
\begin{thm}[{\cite[Theorem 3.4]{SW20}}]\label{thm 3.4}
If $(A,B,C)$ is a spectral interpolation of triple $*$-semisimple Banach $*$-algebras relative to a quasi-Hermitian dense $*$-subalgebra $S$ of $A$, then $S_h$ has invariant spectral radius in $(B,C)$. In particular, we have the canonical $*$-isomorphism $C^*(B)\cong C^*(C)$.
\end{thm}
\subsection{Group representation}
Let $G$ be a discrete group. A \textit{unitary representation} of $G$ is a pair $(\pi,\h)$ where $\h$ is a Hilbert space and $\pi:G\rightarrow \mathcal{U}(\h)$ is a group homomorphism into the group of unitary operators on $\h$. It can be extended linearly to the $*$-representation $\pi:\ell^1(G)\rightarrow \B(\h)$. Conversely, any $*$-representation of $\ell^1(G)$ restricts to a unitary representation of $G$. Thus the representation theory of $G$ and $\ell^1(G)$ coincide.

Let $H\leq G$ be a subgroup. Consider the regular representation 
\begin{align*}
\lambda_{G/H}:\ell^1 (G) \rightarrow \B (\ell^2 (G/H))
\end{align*} defined by
\begin{align*}\lambda_{G/H}(f)g(xH) =(f*g)(xH)= \sum_{y\in G} f(y)g(y^{-1}xH)
\end{align*} for all $f\in \ell^1 (G)$, $g\in \ell^2(G/H)$, and $x\in G$. The kernel of $\Ker = \ker (\lambda_{G/H})$ consists of the functions $g\in \ell^1 (G)$ such that 
$\sum_{h\in H}g(xhy) = 0$ for all $x,y\in G$. Let us denote
\begin{align*}
\ell^1(G:H)= \ell^1 (G )/ Ker .
\end{align*} Then the regular representation gives rise to the  faithful $*$-representation
\begin{align*}
\widetilde \lambda_{G/H}:\ell^1(G:H) \rightarrow \B (\ell^2 (G/H)).
\end{align*}
For $f\in\ell^1(G)$, denote by $[f]\in \ell^1(G:H)$ the corresponding class, and by $\widetilde{f}\in \ell^1(G/H)$ the function defined by $\widetilde{f}(xH)=\sum_{h\in H} f(xh)$ for all $x\in G$. Put
\begin{align*}
[\C G] = \{[f]\in\ell^1(G:H): f\in \C G\}.
\end{align*}
\begin{defn}\label{defn GH repn}
A \textit{$(G,H)$-representation} is a unitary representation $(\pi,\h)$ of $G$ such that $\pi(f)=0$ for all $f\in Ker$.
\end{defn}

Remark that $(\pi,\h)$ is a $(G,H)$ representation if and only if $(\pi,\h)$ is a $*$-representation of $\ell^1(G)$ such that $\pi(Ker)=\{0\}$ if and only if  $(\pi,\h)$ induces a $*$-representation of $\ell^1(G:H)$ if and only if  $(\pi,\h)$ induces a $*$-representation of  $[\C G]$.  We will denote by $C^*(G:H)$ the universal $C^*$-enveloping algebra of $\ell^1(G:H)$ (or equivalently $[\C G]$). If we denote by $(\pi_u,\h_u)$ the largest  $(G:H)$ representation of $G$ (in the sense of weak containment), then we have $C^*_{\pi_u} (G) \cong C^* (G:H)$.

\subsection{Pseudo-function algebras}
We will now construct an analogue of pseudo-function $*$-algebras for the pair $(G,H)$. Let $1\leq p\leq q\leq \infty$ be conjugate numbers, i.e. $p^{-1}+q^{-1}=1$. Consider the regular representation 
\begin{align*}
\lambda^p_{G/H}:\ell^1 (G) \rightarrow \B (\ell^p (G/H))
\end{align*} defined by
\begin{align*}\lambda^p_{G/H}(f)g(xH) =(f*g)(xH) \sum_{y\in G} f(y)g(y^{-1}xH)
\end{align*} for all $f\in \ell^1 (G)$, $g\in \ell^p(G/H)$, and $x\in G$. Since its kernel is independent of the choice of $1\leq p\leq \infty$, we get the faithful representation
\begin{align*}
\widetilde \lambda^p_{G/H}:\ell^1(G:H) \rightarrow \B (\ell^p (G/H)).
\end{align*}
We denote by 
\begin{align*}
[f]\in \ell^1(G:H)\mapsto\|f\|_{PF_p} := \|\lambda^p_{G/H}(f)\|
\end{align*} the induced norm on $\ell^1(G:H)$. For $[f]\in \ell^1(G:H)$, its $PF_p^*$-norm is given by 
\begin{align*}
\|f\|_{PF_p^*} = \max \{ \|f\|_{PF_p} , \|f\|_{PF_q} \}.
\end{align*}
Denote by $PF_p^*(G:H)$ the completion of $\ell^1(G:H)$ with respect to $PF_p^*$-norm. Note that $\|f\|_{PF_p} = \|f^*\|_{PF_q}$ for all $f\in \ell^1(G)$, and consequently $PF_p^*(G:H)$ becomes a unital Banach $*$-algebra. For instance, when $p=2$, we have  $PF_2^*(G:H)  \cong C^*_{\lambda_{G/H}} (G)$. The following lemma helps to understand the image of the algebras $PF_p^*(G:H)$.

\begin{lem}\label{lem injective}
Let $1\leq p_1\leq p_2\leq 2$. The identity map on $\ell^1(G:H)$ extends to the norm decreasing injective $*$-homomophisms
\begin{align*}
\ell^1(G:H)\rightarrow PF_1^*(G:H)\rightarrow PF_{p_1}^*(G:H)\rightarrow PF_{p_2}^*(G:H)\rightarrow C^*_{\lambda_{G/H}}(G).
\end{align*} In particular, these Banach $*$-algebras are $*$-semisimple.
\end{lem}
\begin{proof}
Let $q_i$ be the conjugate of $p_i$, i.e. $p_i^{-1} + q_i^{-1}=1$. For any $f\in \ell^1(G)$, we have by complex interpolation
\begin{align*}
\|f\|_{PF^*_{p_1}} \geq \|\lambda_{G/H}^{p_1}(f)\|^{1-\theta} \|\lambda^{q_1}_{G/H}(f)\|^\theta  \geq \|\lambda^{p_2}_{G/H}(f)\|
\end{align*} where $\theta\in [0;1]$ is such that $p_2^{-1} = (1-\theta )p_1^{-1} +\theta q_1^{-1}$. Using the same inequality for $q_2$, we get $\|f\|_{PF^*_{p_1}} \geq \|f\|_{PF^*_{p_2}}$. In other words, the identity on $\ell^1(G:H)$ extends to a norm decreasing $*$-homomorphism $i_{p_1,p_2}:PF_{p_1}^*(G:H)\rightarrow PF_{p_2}^*(G:H)$. Take any $T\in\ker (i_{p_1,p_2})$. Then there is a sequence of functions $f_n\in \C G$ such that $[f_n]\rightarrow T$ in $PF_{p_1}^*$-norm and $[f_n]\rightarrow 0$ in $PF_{p_2}^*$-norm. Seeing $T$ as an operator on $\ell^{p_1}(G/H)$, we get
\begin{align*}
\langle T\delta_{yH},\delta_{xH}\rangle_{\ell^{p_1},\ell^{q_1}}  &=\lim_n \langle f_n*\delta_{yH},\delta_{xH}\rangle_{\ell^{p_1},\ell^{q_1}} 
\\
&= \lim_n \sum_{h\in H} f_n(xhy)
\\
&= \lim_n \langle f_n*\delta_{yH},\delta_{xH}\rangle_{\ell^{p_2},\ell^{q_2}} = 0
\end{align*} for all $x,y\in G$. Since $\{\delta_{xH}:x\in G\}$ is total in $\ell^p(G/H)$ for any $1\leq p<\infty$, we get $T=0$. This proves that $i_{p_1,p_2}$ is injective.
\end{proof}

\section{Main section}
\subsection{Quasi-Hermitian pair}
We are ready to give the definition of quasi-Hermitian pair.
\begin{defn}\label{defn q-Her} 
Let $G$ be a discrete group, and let $H$ be its subgroup. The pair $(G,H)$ is called \textit{quasi-Hermitian} (resp. \textit{quasi-symmetric}) if $[\C G]$ is quasi-Hermitian (resp. quasi-symmetric) in $PF^*_1(G:H)$. The group $G$ is \textit{Hermitian} (resp. \textit{symmetric}) if the pair $(G,\{e\})$ is quasi-Hermitian (resp. quasi-symmetric).
\end{defn}

\begin{example}
If $H$ has finite index in $G$, then the Banach $*$-algebras $PF^*_p(G:H)$ are finite dimensional and Hermitian. In particular, the pair $(G,H)$ is quasi-Hermitian.
\end{example}

\begin{example}\label{exam 1}
Suppose that $H$ is a non-amenable normal subgroup of $G$ such that $Q=G/H$ is quasi-Hermitian. Then $G$ is not quasi-Hermitian, and $(G,H)$ is quasi-Hermitian. Indeed, the quotient map $p:G \rightarrow Q$ induces $*$-isomorphism $\ell^1(G:H)\cong \ell^1(Q)$ that is stable on $[\C G]\cong \C Q$. Since $\ell^1(Q)$ is isometrically embedded in $\B(\ell^1(G/H))$ and $\B(\ell^\infty (G/H))$, we have $\ell^1(G:H)\cong \ell^1(Q) \cong PF^*_1(G:H)$. Thus, $(G:H)$ is quasi-Hermitian if and only if $[\C G]$ is quasi-Hermitian in $PF^*_1(G:H)$ if and only if $\C Q$ is quasi-Hermitian in $\ell^1(Q)$ if and only if $Q$ is quasi-Hermitian.
\end{example}

\begin{prop}\label{prop PF spectral interpolation}
Let $1\leq p_1< p_2<p_3\leq  2$. Then  
\begin{align*}
(PF^*_{p_1}(G:H),PF^*_{p_2}(G:H),PF^*_{p_3}(G:H))
\end{align*} is a spectral interpolation of triple Banach $*$-algebras relative to $PF^*_{p_1}(G:H)$.
\end{prop}
\begin{proof}
The proof is a verbatim of \cite[Proposition 4.5]{SW20}. The spectral radius inequality part is a well known consequence of complex interpolation, and Lemma \ref{lem injective} shows that $PF^*_{p_1}(G:H)\subseteq PF^*_{p_2}(G:H)\subseteq PF^*_{p_3}(G:H)$ is a nested triple of $*$-subalgebras.

\end{proof}

\begin{thm}\label{thm char q-Her} Let $G$ be a discrete group, and let $H\leq G$ be a subgroup.
The following are equivalent.
\begin{enumerate}[(i)]
\item $(G,H)$ is quasi-symmetric.
\item $(G,H)$ is quasi-Hermitian.
\item For all $f=f^*\in \C G$, we have $r([f],PF^*_1(G:H)) = \|\lambda_{G/H}(f)\|$.
\item For all $f=f^*\in \C G$,  $r([f],PF^*_p(G:H))$ is independent of $p\in [1,2]$.
\item For all $f\in \C G$, we have $\Sp ([f],PF^*_1(G:H)) = \Sp ([f],C^*_{\lambda_{G/H}}(G))$.
\item For all $f\in \C G$,  $\Sp ([f],PF^*_1(G:H))$ is independent of $p\in [1,2]$. 
\end{enumerate}
\end{thm}
\begin{proof}
The direction $(i)\Rightarrow (ii)$ is the easier part of Shirali-Ford's theorem (see \cite[Remark p.225]{BD73}). The directions
$(iii)\Rightarrow (iv)$, $(v)\Rightarrow (vi)$, $(vi)\Rightarrow (i)$ are clear. The direction $(iv)\Rightarrow (v)$ is thanks to Barnes-Hulanicki's Theorem \cite[p.329]{Bar90}.
It only remains to prove $(ii)\Rightarrow (iii)$. The proof is an adaptation of \cite[Proposition 4.7]{SW20}. Assume $(ii)$. 
Then $(PF^*_{1}(G:H),PF^*_{p}(G:H),C^*_{\lambda_{G/H}}(G))$ is a triple of spectral interpolation relative to the  quasi-Hermitian $*$-subalgebra $[\C G]\subseteq PF^*_{1}(G:H)$ for any $1<p<2$. Thus, 
\begin{align*}
r([f],PF^*_p(G:H))=\|\lambda_{G/H}(f)\|\quad \text{for all}\quad f\in \C G
\end{align*} by Theorem \ref{thm 3.4}. Take any $f=f^*\in \C G$. For any $n\in \N$ and $\varepsilon >0$, there exists $g\in \C G$ such that $\|\widetilde{g}\|_{\ell^1(G/H)}\leq 1$ and 
\begin{align*}(1-\varepsilon)
\|f^{(n)}\|_{PF^*_1(G:H)}  \leq
\|\lambda_{G/H}^{1} (f^{(n)})\widetilde{g} \|_{\ell^1(G/H)} = 
\|\lambda_{G/H}^{1} (f^{(n-1)})\widetilde{f*g} \|_{\ell^1(G/H)}.
\end{align*}
Let $S$ be a symmetric finite subset of $G$ containing the support of $f^{(n)}*g$. Then $\lambda^1_{G/H}(f^{(n)}) \widetilde{g}$ is supported on $[S] = \{sH\in G/H: s\in S\}$ and
\begin{align*}(1-\varepsilon)
\|f^{(n)}\|_{PF^*_1(G:H)}
&\leq\|\lambda_{G/H}^{1} (f^{(n-1)})(\widetilde{f*g}) \n1_{[S]}\|_{\ell^1(G/H)} 
\\
&\leq\|f^{(n-1)}\|_{PF_{p}^*(G:H)} \|\widetilde{f*g}\|_{\ell^p(G/H)} \|\n1_{[S]}\|_{\ell^q(G/H)}
\\
&\leq\|f^{(n-1)}\|_{PF_{p}^*(G:H)}  \|\widetilde{f*g}\|_{\ell^p(G/H)} |S|^{1/q}
\\
&\leq\|f^{(n-1)}\|_{PF_{p}^*(G:H)}  \|f\|_{PF^*_1(G:H)}  |S|^{1/q}.
\end{align*} Taking the $n$th root on  both sides and $\limsup$ over $p\rightarrow 1$, we get
\begin{align*}(1-\varepsilon)^{1/n}
\|f^{(n)}\|_{PF^*_1(G:H)}^{1/n} &\leq \limsup_{p\rightarrow 1}\|f^{(n-1)}\|_{PF_{p}^*(G:H)}^{1/n}  \|f\|_{PF^*_1(G:H)}^{1/n}
\\
&\leq \|f^{(n-1)}\|_{PF_{p_0}^*(G:H)}^{1/n}  \|f\|_{PF^*_1(G:H)}^{1/n} +\varepsilon
\end{align*} for some $1<p_0$. Now letting $n\rightarrow \infty$  and then $\varepsilon\rightarrow 0$, we get
\begin{align*}
r([f],PF_1^*(G:H)) \leq r([f],PF_{p_0}^* (G:H)) = \|\lambda_{G/H}(f)\|.
\end{align*} Since the converse inequality is always true, the above is equality.

\end{proof}

\subsection{Co-amenable subgroup}
Co-amenability was introduced in \cite{Eym72} for a pair $(G,H)$ of locally compact group $G$ and its closed subgroup $H$. It extends the usual amenability. Basic examples are when the quotient $G/H$ is compact, and when $H$ is normal and the quotient group $G/H$ is amenable. We recall the definition.
\begin{defn}\label{defn co-amen}
Let $G$ be a discrete group, and let $H\leq G$ be a subgroup. We say that $H$ is \textit{co-amenable} in $G$ if there exists a left $G$-invariant mean on $\ell^\infty (G/H)$.
\end{defn}

The usual amenability is characterized in many different ways. One of them is the canonical equality of the full group $C^*$-algebra $C^*(G)$ and the reduced group $C^*$-algebra $C^*_\lambda (G)$. A partial analogue of this result is observed for co-amenability.
\begin{prop}\label{thm char co-amen}
Let $G$ be a discrete group and let $H\leq G$ be a subgroup. Consider the following statements.
\begin{enumerate}[(i)]
\item The canonical map $C^*(G:H)\twoheadrightarrow C^*_{\lambda_{G/H}} (G)$ is injective.
\item The canonical map $C^*(PF^*_{1}(G:H))\twoheadrightarrow C^*_{\lambda_{G/H}} (G)$ is injective.
\item The subgroup $H$ is co-amenable in $G$.
\end{enumerate} We have $(i)\Rightarrow(ii)\Rightarrow (iii)$. If $H$ is normal, we have $(iii)\Rightarrow (i)$.
\end{prop}

\begin{proof}
The implication $(i)\Rightarrow (ii)$ follows from the surjective $*$-homomorphisms
\begin{align*}
C^*(G:H) \twoheadrightarrow
 C^*(PF^*_1(G:H))\twoheadrightarrow
 C^*_{\lambda_{G,H}}(G).
\end{align*}

Let us prove $(ii)\Rightarrow (iii)$. Suppose that the map $C^*(PF^{*}_1(G:H))\rightarrow C^*_{\lambda_{G/H}} (G)$ is injective, hence $*$-isomorphism. Consider the character $\sigma:[f]\in \ell^1(G:H)\mapsto \sum_{x\in G} f(x)$. Since we have 
\begin{align*}
|\sigma([f]) |= |\sum_{x\in G} f(x)| \leq \sum_{xH\in G/H} |\sum_{h\in H}f(xh)| = \|\lambda_{G/H}^1 (f)\delta_H\|_1 \leq \|f\|_{PF^*_1(G:H)}.
\end{align*} $\sigma$ extends to a character on $C^*(PF^*_1(G:H)) = C^*_{\lambda_{G/H}} (G)$. It also extends to a state on $\B(\ell^2(G/H))$ which we also denote by $\sigma$. Observe that
$
\sigma(\lambda_{G/H}(x)) =\sigma([\delta_x]) = 1
$ for all $x\in G$. In particular, the left translation operators  $\{\lambda_{G/H}(x):x\in G\}$ are in the multiplicative domain of $\sigma$, that is $
\sigma(\lambda_{G/H}(x)^*\lambda_{G/H}(x)) =\sigma(\lambda_{G/H}(x))^*\sigma(\lambda_{G/H}(x))$ and $
\sigma(\lambda_{G/H}(x)\lambda_{G/H}(x)^*) =\sigma(\lambda_{G/H}(x))\sigma(\lambda_{G/H}(x))^*$ for all $x\in G$. It follows from Bimodule Property (cf. \cite[Proposition 1.5.7]{BO08}) that 
\begin{align*}
\sigma(\lambda_{G/H}(x) T\lambda_{G/H}(x^{-1})) = \sigma(\lambda_{G/H}(x)) \sigma(M)\sigma (\lambda_{G/H}(x^{-1})) = \sigma(T)
\end{align*} for all $T\in \B(\ell^2 (G/H))$. Recall that the pointwise multiplication operators give faithful $*$-representation 
\begin{align*}
f\in\ell^\infty (G/H)\mapsto M_f\in \B(\ell^2 (G/H))
\end{align*} and we have $M_{\lambda_x f} = \lambda_{G/H}(x) M_f\lambda_{G/H}(x^{-1})$. Now, it is clear that $\sigma$ gives a $G$-invariant state on $\ell^\infty (G/H)$.

For the direction $(iii)\Rightarrow (i)$, when $H$ is normal and co-amenable in $G$, the quotient group $Q=G/H$ is amenable. Moreover, we have the canonical isomorphisms $\ell^1(Q)\cong \ell^1(G:H)$, $C^*(Q)\cong C^*(G:H)$, and $C^*_{\lambda}(Q)\cong C^*_{\lambda_{G/H}}(G)$. Thus $(iii)\Rightarrow (i)$ follows.
\end{proof}
Now, we are ready to prove Theorem \ref{thm A}.
\begin{proof}[Proof of Theorem \ref{thm A}] Assume that $(G,H)$ is a quasi-Hermitian pair. Then $[\C G]$ is spectral in $(PF^*_1(G:H),C^*_{\lambda_{G/H}} (G))$ by Theorem \ref{thm char q-Her}. Proposition \ref{prop 2.6}  yields that we have canonical $*$-isomorphism $C^*(G:H)\rightarrow C^*_{\lambda_{G/H}} (G)$. Then $H$ is co-amenable in $G$ by Proposition \ref{thm char co-amen}.
\end{proof}

\begin{cor}\label{cor1}
Quasi-Hermitian discrete groups are amenable.
\end{cor}
\begin{proof}
Choose $H=\{e\}$ and apply Theorem \ref{thm A}.
\end{proof}

\bibliographystyle{amsalpha}
\bibliography{mybibfile}
\end{document}